\documentclass[11pt]{article}

\usepackage{amsmath}
\usepackage{amsthm}
\usepackage{amssymb}

\newtheorem{theorem}{Theorem}[section]
\newtheorem{definition}[theorem]{Definition}
\newtheorem{example}[theorem]{Example}

\newtheorem{corollary}[theorem]{Corollary}
\newtheorem{remark}[theorem]{Remark}

\newtheorem{proposition}[theorem]{Proposition}
{

\begin{document}
\title{Best proximity pair results for relatively nonexpansive mappings in geodesic spaces}
\author{Aurora Fern\'andez-Le\'on$^{1}$, Adriana Nicolae$^{2,3}$}
\date{}
\maketitle

\begin{center}
{\footnotesize
$^{1}$Dpto. de An\'alisis Matem\'atico, Universidad de Sevilla, P.O. Box 1160, 41080-Sevilla, Spain\\
${}^2$ Department of Mathematics, Babe\c{s}-Bolyai University, Kog\u{a}lniceanu 1, 400084 Cluj-Napoca, Romania\\
${}^3$ Simion Stoilow Institute of Mathematics of the Romanian Academy, Research group of the project PD-3-0152, P.O. Box 1-764, RO-014700 Bucharest, Romania\\
E-mail: aurorafl@us.es, anicolae@math.ubbcluj.ro
}
\end{center}

\begin{abstract}
Given $A$ and $B$ two nonempty subsets in a metric space, a mapping $T : A \cup B \rightarrow A \cup B$ is relatively nonexpansive if $d(Tx,Ty) \leq d(x,y) \text{ for every } x\in A, y\in B.$ A best proximity point for such a mapping is a point $x \in A \cup B$ such that $d(x,Tx)=\text{dist}(A,B)$. In this work, we extend the results given in [A.A. Eldred, W.A. Kirk, P. Veeramani, Proximal normal structure and relatively nonexpansive mappings, Studia Math. 171, 283--293 (2005)] for relatively nonexpansive mappings in Banach spaces to more general metric spaces. Namely, we give existence results of best proximity points for cyclic and noncyclic relatively nonexpansive mappings in the context of Busemann convex reflexive metric spaces. Moreover, particular results are proved in the setting of CAT($0$) and uniformly convex geodesic spaces. Finally, we show that proximal normal structure is a sufficient but not necessary condition for the existence in $A \times B$ of a pair of best proximity points.\\

\vskip.27cm


\vskip.25cm


\end{abstract}

\section{Introduction}
Although metric fixed point theory is primary concerned with the existence of fixed points of mappings that satisfy certain restrictions, there exist many other related problems that have attracted a high amount of interest from researchers in the area. One of such problems consists in studying the existence of approximate solutions of the equation $x=Tx$ in the absence of fixed points of the mapping $T$. A point $x\in X$ is said to be an approximate solution of the equation $x=Tx$ if $x$ is ``close'' to $Tx$ is some sense. Depending on the considered closeness condition between $x$ and $Tx$, results of different nature have been obtained in the literature. One classical result in this direction due to Ky Fan \cite{fan} states that if $A$ is a compact, convex and nonempty subset of a  locally convex Hausdorff topological vector space $X$ and $T$ is a continuous mapping from $A$ to $X$, then there exists a point $x\in A$ such that $d(x,Tx)=d(Tx,A)$, where $d$ is the semi-metric induced by a continuous semi-norm defined on $X$. If, instead of considering the condition $d(x,Tx)=d(Tx,A)$, one requires that $x$ is an absolute optimal approximate solution, that is, $d(x,Tx)=\text{dist}(A,B)$ either for non-self mappings $T : A \rightarrow B$ or for mappings $T : A \cup B \rightarrow A \cup B$ such that $T(A) \subseteq B$, $T(B) \subseteq A$ or $T(A) \subseteq A$, $T(B) \subseteq B$, existence, uniqueness and convergence results for such points are known as best proximity point theorems. Note that the notion of best proximity point also refers to such a type of approximate solution. In the present work we mainly focus on the study of best proximity points for certain self-mappings $T : A \cup B \rightarrow A \cup B$ satisfying the above inclusion relations. The first results concerning such mappings were given by Kirk, Srinivasan and Veeramani \cite{kive} in 2003. More precisely, it was proved that if $A$ and $B$ are two nonempty and closed subsets of a complete metric space, $T : A \cup B \rightarrow A \cup B$ is such that $T(A) \subseteq B$, $T(B) \subseteq A$ and there exists $k\in (0,1)$ such that
$$d(Tx,Ty) \leq k d(x,y) \text{ for every } x\in A, y\in B,$$ then $A\cap B$ contains a fixed point of $T$.

In the last years, many generalizations of this problem have appeared under the assumption $A\cap B = \emptyset$. In this respect, weaker metric conditions have been considered for the mapping $T$. This is, for instance, the case of cyclic contractions \cite{elve,suzu,esfe2}, cyclic Meir-Keeler contractions \cite{dibar,piakeeler} or relatively nonexpansive mappings \cite{elvek,espi,sankar}. Relatively nonexpansive mappings were introduced by Eldred, Kirk and Veeramani \cite{elvek} in the following way: a self-mapping $T : A \cup B \rightarrow A \cup B$ is relatively nonexpansive if
$$d(Tx,Ty) \leq  d(x,y) \text{ for every } x\in A, y\in B.$$
If, in addition, $T(A) \subseteq B$ and $T(B) \subseteq A$ then $T$ is said to be cyclic. Likewise, if $T(A) \subseteq A$ and $T(B) \subseteq B$, then $T$ is called noncyclic.

In \cite{elvek}, several best proximity point results were given in Banach spaces for cyclic and noncyclic relatively nonexpansive mappings. While many generalizations of best proximity point results from the linear setting to metric spaces have appeared in the literature (see, among others, the works \cite{esfe2,fer} on cyclic contractions in metric spaces), no result has been given for relatively nonexpansive mappings in a nonlinear setting. Here we address this problem extending results proved in \cite{elvek} in the context of Banach spaces to Busemann convex reflexive metric spaces. We also give more particular results in the setting of CAT($0$) and uniformly convex geodesic spaces. Furthermore, we prove an analogue of a result due to Karlovitz \cite{kar} showing that proximal normal structure is a sufficient but not necessary condition for the existence in $A \times B$ of a pair of best proximity points.

\section{Preliminaries}

In this section we compile the main concepts and results we will
work with along this paper. We begin with some basic definitions
and notations that are needed. Let $(X,d)$ be a metric space and consider
$A$ and $B$ two subsets of $X$. Define
\begin{align*}
d(x,A)= & \inf \{d(x,y) : y \in A\};\\
\text{dist}(A,B)= &\inf \{d(x,y) : x\in A,y\in B\};\\
\delta(x,A)= &\sup \{d(x,y) : y \in A\};\\
\delta(A,B)= & \sup \{d(x,y) : x \in A, y \in B\}.
\end{align*}
From now on, $B (a,r)$ denotes the closed ball in the space $X$
centered at $a\in X$ with radius $r>0$.

The \emph{metric projection} $P_A$ onto $A$ is the mapping
$$P_A (x) = \{z\in A : d(x,z) = \text{dist}(x,A)\} \text{ for every } x \in X.$$
When this mapping is well-defined and singlevalued we use the same notation $P_A(x)$ to denote
the unique point belonging to this set.

In the sequel, we say that a pair of sets $(A,B)$ has a property if each of the sets $A$ and $B$ has this property. For instance, we say that the pair $(A,B)$ is closed and bounded if $A$ and $B$ are both closed and bounded. A very important property in this paper for a pair of sets is the one of proximity.

\begin{definition} A pair $(A,B)$ of subsets of a metric space is said
to be proximal if for each $(a,b) \in A\times B$ there exists $(a^\prime, b^\prime) \in A\times B$
such that $d(a,b^\prime)=d(a^\prime,b)=\text{dist}(A,B)$.
\end{definition}

In this context, given a pair of sets $(A,B)$ in a metric space, we say that the point $p \in A$ is a proximal point of $q \in B$ (with respect to $A$ and $B$) if $d(p,q)=\mbox{dist}(A,B)$. Then, $(p,q)$ is also called pair of proximal points.

In this paper we will
mainly work with geodesic spaces. A metric space
$(X,d)$ is said to be a {\it (uniquely) geodesic space} if every two points
$x$ and $y$ of $X$ are joined by a {\it (unique) geodesic}, i.e, a map
$c:[0,l]\subseteq {\mathbb R}\to X$ such that $c(0)=x$, $c(l)=y$,
and $d(c(t),c(t^{\prime}))=|t-t^{\prime}|$ for all $t,t^{\prime}
\in [0,l]$. The image $c([0,l])$ of such a geodesic forms a {\it geodesic segment} which joins $x$ and $y$ and it is not necessarily unique. If no confusion arises, we use $[x,y]$ to denote a geodesic segment
joining $x$ and $y$. A point $z$ in $X$ belongs
to a geodesic segment $[x,y]$ if and only if there exists $t\in [0,1]$ such that $d(x,z)=td(x,y)$ and $d(y,z)=(1-t)d(x,y)$ and we write $z=(1-t)x+ty$ for simplicity. Notice that this point may not be unique. When $t=\frac{1}{2}$, we often use the notation $\frac{x+y}{2}$ to denote $\frac{1}{2}x+\frac{1}{2}y$. Any Banach space is a
geodesic space with usual segments as geodesic segments. A subset $A$ of a geodesic space $X$ is said to be
{\it convex} if any geodesic segment that joins each pair of points
$x$ and $y$ of $A$ is contained in $A$. A {\it geodesic triangle} $\triangle(x,y,z)$ in $X$ consists of three points $x,y,z \in X$ (the \emph{vertices} of $\triangle$) and three geodesic segments joining each pair of vertices (the {\it edges} of $\triangle$). For more about geodesic
spaces the reader can check \cite{brha,bubu,papa}.

A \emph{metric} $d \colon X \times X \rightarrow {\mathbb R}$ is said to be \emph{convex} if for any $x, y, z \in X$ one has
$$d(x, (1 - t)y + tz) \leq (1 - t)d(x, y) + td(x, z) \text{ for all } t \in [0, 1].$$

A geodesic space $(X,d)$ is {\it Busemann convex} (introduced in \cite{bu}) if given any pair of geodesics $c_1 : [0, l_1] \to X$ and $c_2 : [0,l_2] \to X$ one has
$$
d(c_1(tl_1),c_2(tl_2)) \le (1-t)d(c_1(0),c_2(0)) + td(c_1(l_1),c_2(l_2)) \mbox{ for all } t \in [0,1].
$$
It is well-known that Busemann convex spaces are uniquely geodesic and with convex metric.

Given two geodesic segments $[x,z]$ and $[y,w]$ in a uniquely geodesic space $(X,d)$, we say that $[x,z]$ is \emph{parallel} to $[y,w]$, and we denote it by $[x,z] \| [y,w]$, if $d(x,y)=d(m_1,m_2)=d(z,w)$, where $m_1$ and $m_2$ are the midpoints of $[x,z]$ and $[y,w]$, respectively (that is, $m_1=\frac{x+z}{2}$ and $m_2=\frac{y+w}{2}$). The following property was given by Busemann in \cite{busem}. For the convenience of the reader we include a proof of this fact.
\begin{proposition}\label{para}
Let $x,y,z,w$ be four points in a Busemann convex geodesic space. Suppose $[x,z] \|[y,w]$. Then, $[x,y] \| [z,w]$.
\end{proposition}

\begin{proof} Let $m_1=\frac{x+z}{2}$, $m_2=\frac{y+w}{2}$, $m_3=\frac{x+y}{2}$ and $m_4=\frac{z+w}{2}$. Let $r=d(x,y)=d(z,w)=d(m_1,m_2)$. By using the Busemann convexity of the space we have that $d(m_1,\frac{y+z}{2})\leq r/2$, $d(m_1,\frac{w+x}{2})\leq r/2$, $d(m_2,\frac{y+z}{2})\leq r/2$ and $d(m_2,\frac{w+x}{2})\leq r/2$, which implies
$$r = d(m_1,m_2) \leq d\bigg{(}m_1,\frac{y+z}{2}\bigg{)}+ d\bigg{(}\frac{y+z}{2},m_2\bigg{)}\leq \frac{r}{2}+\frac{r}{2}=r$$ and
$$r = d(m_1,m_2) \leq d\bigg{(}m_1,\frac{w+x}{2}\bigg{)}+ d\bigg{(}\frac{w+x}{2},m_2\bigg{)}\leq \frac{r}{2}+\frac{r}{2}=r.$$
Consequently, $\frac{y+z}{2}=\frac{m_1+m_2}{2}=\frac{w+x}{2}.$ Let $m_5=\frac{m_1+m_2}{2}$. Again by the Busemann convexity,
$d(m_4,m_5)\leq 1/2 d(x,z)$ and $d(m_4,m_5)\leq 1/2 d(w,y)$, which implies $d(m_4,m_5) \leq 1/2 \min\{d(x,z),$ $d(w,y)\}$. Since $m_5$ is also the midpoint between $m_3$ and $m_4$, $d(m_3,m_4) \leq \min \{d(x,z),d(w,y)\}$. Suppose now that $d(x,z) < d(y,w)$. Then $d(m_3,m_4) \leq d(x,z)<d(y,w)$. Since $[m_3,m_4] \| [y,w]$, we can proceed similarly to get $p_1=\frac{m_3+y}{2}$ and $s_1=\frac{m_4+w}{2}$ such that $d(p_1,s_1)\leq d(x,z) <d(y,w)$. By repeating the process, we obtain the sequences $\{p_n\},\{s_n\} \subseteq X$, where, for $n\geq 2$, $p_n=\frac{y+p_{n-1}}{2}$ and $s_n=\frac{w+s_{n-1}}{2}$, with $d(s_n,w)=d(p_n,y)=r/2^{n+1}$ and $d(p_n,s_n)\leq d(x,z)<d(y,w).$

Since $d(y,w)\leq d(y,p_n)+d(p_n,s_n)+d(s_n,w)$ for every $n\in \mathbb{N}$, we may take superior limit in the previous inequality to get $d(y,w)\leq d(x,z) < d(y,w)$, which is a contradiction and the result follows. \end{proof}

In the sequel, we will also need the notion of uniformly convex geodesic space (see also \cite[pg. 107]{gore}).
A geodesic metric space $(X,d)$ is said to be {\it uniformly
convex} if for any $r>0$ and any $\varepsilon \in
(0,2]$ there exists $\delta \in (0,1]$ such that for all $a, x,
y\in X$ with $d(x,a)\le r$, $d(y,a)\le r$ and $d(x,y)\geq
\varepsilon r$,
$$
d(m,a)\le (1-\delta)r,
$$
where $m$ stands for a midpoint of $x$ and $y$. A mapping
$\delta:(0,+\infty)\times (0,2]\to (0,1]$ providing such a
$\delta=\delta (r,\varepsilon )$ for a given $r>0$ and
$\varepsilon \in (0,2]$ is called a {\it modulus of uniform
convexity} of $X$. If $\delta$ decreases with $r$ (for each fixed
$\varepsilon$) we say that $\delta$ is a monotone modulus of
uniform convexity of $X$ (introduced in \cite{leus}). If $\delta$ is lower semicontinuous from the right with respect to $r$ (for each fixed
$\varepsilon$), then we say $\delta$ is a lower semicontinuous from the right modulus of uniform convexity of $X$.

If in the above definition we drop the uniformity conditions then
we find the notion of {\it strict convexity} in metric spaces. Consequently, every uniformly convex geodesic space is strictly convex. Moreover,
it is easy to see that every Busemann convex metric space is strictly convex \cite{bora} and that strictly convex metric spaces are
uniquely geodesic.

A very important class of geodesic spaces are CAT($0$) spaces, that is, metric spaces of
nonpositive curvature in the sense of Gromov. These spaces play an essential role in several areas of mathematics \cite{brha} and find
applications in other branches of science such as biology and computer science \cite{bart,ow}. CAT($0$) spaces are defined in terms of comparison with ${\mathbb E}^2$, the Euclidean plane, as follows: given $(X,d)$ a geodesic metric space, a {\it comparison triangle}
for a geodesic triangle $\triangle (x_1,x_2,x_3)$ in $(X,d)$ is a
triangle $\triangle(\bar{x}_1,\bar{x}_2,\bar{x}_3)$ in ${\mathbb E}^2$ such
that $d_{{\mathbb E}^2}(\bar{x}_i,\bar{x}_j)=d(x_i,x_j)$ for $i,j\in\{
1,2,3\}$. Such a comparison triangle always exists
in ${\mathbb E}^2$ and is unique up to isometry. A geodesic triangle $\triangle$ in $X$ is said to satisfy the {\it CAT(0) inequality} if, given $\bar{\triangle}$ a comparison
triangle in ${\mathbb E}^2$ for $\triangle$, for all $x,y\in \triangle$
$$
d(x,y)\le d_{{\mathbb E}^2}(\bar{x},\bar{y}),
$$
where $\bar{x},\bar{y}\in\bar{\triangle}$ are the
comparison points of $x$ and $y$, respectively. A geodesic space $X$ is a {\it CAT(0) space} if all its geodesic triangles satisfy the
CAT($0$) inequality.

The following four point condition was used by Berg and Nikolaev \cite{beni} to characterize
CAT($0$) spaces.
\begin{theorem}\label{charact} Let $(X,d)$ be a geodesic space. $X$ is a CAT(0) space if and
only if for every $x, y, z, p \in X$,
$$d(x, z)^2 + d(y, p)^2 \leq d(x, y)^2 + d(y, z)^2 + d(z, p)^2 + d(p,x)^2.$$
\end{theorem}

In complete CAT($0$) spaces, the metric projection onto closed and convex subsets behaves as in Hilbert spaces in a certain sense.

\begin{proposition}[\cite{brha}, Proposition 2.4, p. 176] \label{projection}
Let $X$ be a complete CAT(0) space, $x\in X$ and $C\subset X$
nonempty closed and convex. Then
the following facts hold:
\begin{enumerate}

\item The metric projection $P_C(x)$ of $x$ onto $C$ is a
singleton.
\item If $y\in [x,P_C(x)]$, then $P_C(x)=P_C(y)$.
\item If $x\notin C$ and $y\in C$ with $y\neq P_C(x)$ then
$\angle_{P_C(x)}(x,y)\geq \pi/2$.
\item The mapping $P_C$ is a nonexpansive retraction from $X$ onto $C$. Further, the mapping $H : X \times [0,1] \rightarrow X$ associating to $(x,t)$ the point at distance $td(x,P_C(x))$ on the geodesic $[x,P_C(x)]$ is a continuous homotopy from the identity map of $X$ to $P_C$.
\end{enumerate}
\end{proposition}
\noindent For a thorough treatment of CAT($0$) spaces and related topics the reader can check \cite{brha,gr}.

In the next section we will also work with reflexive metric spaces which extend the notion of reflexivity from Banach to metric spaces. A geodesic metric space $X$ is said to be {\it reflexive} if for every decreasing chain
$ \{C_\alpha\} \subset X$ with $\alpha \in I$ such that $C_\alpha$ is closed convex
bounded and nonempty for all $\alpha\in I$ we have that $\displaystyle{\bigcap_{\alpha \in I} C_\alpha\neq \emptyset}$. Notice that every complete uniformly convex metric space with either a monotone or lower
semicontinuous from the right modulus of uniform convexity is reflexive (see \cite{leus1,esfe1}). Also note that a reflexive and Busemann convex geodesic space is complete (see \cite[Lemma 4.1]{esfe3}). Moreover, in such a context the metric projection onto closed and convex subsets is well-defined and singlevalued.

Next we give the definition of relatively nonexpansive mapping on the union of two sets.

\begin{definition} Suppose $A$ and $B$ are two nonempty subsets of a metric space $X$. A mapping $T : A \cup B \rightarrow A \cup B$ is relatively nonexpansive
if $d(Tx,Ty) \leq d(x,y)$ for all $x\in A$ and $y\in B$.
\end{definition}
We say that a relatively nonexpansive mapping $T$ is {\it cyclic} if $T(A) \subseteq B$ and $T(B) \subseteq A$ and {\it noncyclic} if $T(A) \subseteq A$ and $T(B) \subseteq B$.

In \cite{elvek}, the notion of proximal normal structure was introduced as a counterpart of the well-known concept of normal structure. This concept plays the same role for relatively nonexpansive mappings as normal structure plays for nonexpansive mappings. We state below this notion in the setting of geodesic metric spaces.

\begin{definition} A convex pair $(K_1,K_2)$ in a geodesic space is said to
have proximal normal structure if for any closed bounded convex proximal
pair $(H_1,H_2) \subseteq (K_1,K_2)$ for which $\text{dist}(H_1,H_2) = \text{dist}(K_1,K_2)$ and
$\delta(H_1,H_2) > \text{dist}(H_1,H_2)$, there exists $(x_1, x_2) \in  H_1 \times H_2$ such that
$$\delta(x_1,H_2) < \delta(H_1,H_2) \quad \mbox{and} \quad \delta(x_2,H_1) < \delta(H_1,H_2).$$
\end{definition}
As in the linear case, a pair $(K,K)$ has proximal normal structure if and only if
$K$ has normal structure in the sense of Brodski and Milman \cite{brmi}.

\section{Main results}

Given a pair of sets $(A,B)$ in a metric space $X$, let $A_0$ and $B_0$ be the subsets defined as follows:
$$A_0=\{x \in A : d(x,y^{\prime})= \text{dist}(A,B) \text{ for some } y^{\prime} \in B\},$$
$$B_0=\{y \in B : d(x^{\prime},y)= \text{dist}(A,B) \text{ for some } x^{\prime} \in A\}.
$$
\begin{proposition}\label{tec} Let $X$ be a reflexive and Busemann convex metric
space and let $(A,B)$ be a nonempty closed convex pair of subsets in $X$. Suppose additionally $B$ is bounded. Then $A_0$ and $B_0$ are closed, convex,  bounded and nonempty.

\end{proposition}

\begin{proof} First we see that $B_0$ is closed, convex, bounded and nonempty. Given any real number $\varepsilon >0$, consider the set
$$
A'_\varepsilon =\{x \in X : d(x,A)\leq \text{dist}(A,B) + \varepsilon\}.$$
It is easy to see that $A'_\varepsilon$ is nonempty and closed. Moreover, $A'_\varepsilon$ is convex. Let $x$ and $y$ be two points in $A'_\varepsilon$ and $m$ the midpoint between them. Since the space is Busemann convex, we have that
\begin{align*}
d(m,A) & =d(m,P_A(m))\leq d\bigg{(}m,\frac{P_A(x)+P_A(y)}{2}\bigg{)}\\
&\leq \max \{d(x,A),d(y,A)\}\leq \text{dist}(A,B)+\varepsilon.
\end{align*}
Thus $m \in A'_\varepsilon$.

Now consider the set $A_\varepsilon=A'_\varepsilon \cap B$. It is immediate that $A_\varepsilon$ is closed, convex and bounded. Moreover, by definition of $A'_\varepsilon$, $A_\varepsilon$ is also nonempty. Then, by means of the reflexivity of the space, we conclude that $\cap_{\varepsilon >0} A_\varepsilon \neq \emptyset$. Since $B_0=\displaystyle{\cap_{\varepsilon >0} A_\varepsilon}$, we see that $B_0$ is closed, convex, bounded and nonempty.
Notice that $A_0$ is bounded and nonempty since $B_0$ is so. The fact that $A_0$ is also closed and convex follows by a straightforward verification. \end{proof}

\begin{remark} Notice that in the previous result we just need one of the sets $A$ and $B$ to be bounded, no matter which of them.
\end{remark}

\begin{theorem}\label{generaliza} Let $X$ be a reflexive and Busemann convex metric
space and let $(A,B)$ be a nonempty closed convex pair of subsets of $X$ such that $A$ is bounded. Let
$T : A \cup B \rightarrow A \cup B$ be a cyclic relatively nonexpansive mapping. Suppose $(A,B)$ has proximal normal structure. Then there exists a pair $(x,y) \in A \times B$ such that $d(x,Tx)=d(y,Ty)=\emph{dist}(A,B)$.
\end{theorem}

\begin{proof} This result follows by applying similar patterns as in the proof of Theorem 2.1 in \cite{elvek}. However, in this more general setting, several changes and new techniques must be considered to get the result.
From Proposition \ref{tec} we have that $(A_0,B_0)$ is closed, convex, bounded and nonempty. Moreover, we may notice that this pair is also proximal and satisfies $\text{dist}(A_0,B_0)=\text{dist}(A,B)$. It is easy to see that $T(A_0) \subseteq B_0$ and $T(B_0) \subseteq A_0$. Since $(A,B)$ has proximal normal structure, so does $(A_0, B_0)$.

Consider the family $\Gamma$ of sets $F \subseteq A_0 \cup B_0$ such that $F \cap A_0$ and $F \cap B_0$ are closed, convex and nonempty and satisfy
$T(F \cap A_0) \subseteq F \cap B_0$ and $T(F \cap B_0) \subseteq F \cap A_0$, $\text{dist}(F \cap A_0,F \cap B_0)=\text{dist}(A_0,B_0)$ and the pair $(F \cap A_0, F \cap B_0)$ is proximal. Since $A_0 \cup B_0 \in \Gamma$, we have that $\Gamma \neq \emptyset$.

Let $\{F_\alpha\}_{\alpha \in I}$ be a decreasing chain in $\Gamma$. We see that $F_0= \cap_{\alpha \in I} F_\alpha \in \Gamma$. Since
$F_0 \cap A_0=\cap_{\alpha \in I} (F_\alpha \cap A_0)$ and $X$ is reflexive, we have that $F_0 \cap A_0$ is closed, convex and nonempty. Similarly  $F_0 \cap B_0$ is closed, convex and nonempty. It can be easily proved that $T(F_0 \cap A_0) \subseteq F_0 \cap B_0$ and $T(F_0 \cap B_0) \subseteq F_0 \cap A_0$. Thus, to see that $F_0 \in \Gamma$, it remains to prove that the pair $(F_0\cap A_0,F_0 \cap B_0)$ is proximal and $\text{dist}(F_0\cap A_0,F_0\cap B_0)=\text{dist}(A_0,B_0)$. Let $p \in F_0 \cap A_0$. By definition, $p \in F_\alpha \cap A_0$ for every $\alpha \in I$. Since $(A_0,B_0)$ is proximal and $X$ is strictly convex, there exists a unique point $q \in B_0$ such that $d(p,q)=\text{dist}(A_0,B_0)$. Moreover, by using the proximity of $(F_\alpha \cap A_0,F_\alpha \cap B_0)$, we have that there exists a point $q_\alpha \in F_\alpha \cap B_0$ such that $d(p, q_\alpha)=\text{dist}(F_\alpha \cap A_0,F_\alpha \cap B_0)=\text{dist}(A_0,B_0)$ for every $\alpha \in I$. However, since $q_\alpha \in B_0$, we have that $q_\alpha=q$ for every $\alpha \in I$ and therefore $q \in F_0 \cap B_0$. Consequently, $F_0 \in \Gamma$ and applying Zorn's Lemma we obtain a minimal element $K$ in $\Gamma$.

Let $K_1=K \cap A_0$ and $K_2=K \cap B_0$. If $\delta(K_1,K_2)=\text{dist}(K_1,K_2)$, then $d(p,q)=\text{dist}(K_1,K_2)$ for every pair $(p,q)\in K_1 \times K_2$. In particular, $d(x,Tx)=\text{dist}(K_1,K_2)=\text{dist}(A_0,B_0)$ for every $x \in K$ and the result holds.
Suppose now $\delta(K_1,K_2)>\text{dist}(K_1,K_2)$. Since $(A_0,B_0)$ has proximal normal structure, there exists $(y_1,y_2) \in K_1 \times K_2$ and $\lambda \in (0,1)$ such that $\delta(y_1,K_2) \leq \lambda \delta(K_1,K_2)$ and $\delta(y_2,K_1)\leq \lambda \delta(K_1,K_2)$. Since $K \in \Gamma$, $(K_1,K_2)$ is proximal. In fact, we may notice by the strict convexity of $X$ that for every point $p\in K$ there exists only one point $q \in K$ such that $d(p,q)=\text{dist}(K_1,K_2)$. Let $(y'_1, y'_2) \in K_1 \times K_2$ such that $d(y_1,y'_2)=d(y'_1,y_2)=\text{dist}(K_1,K_2)$. Denote by $x_1$ and $x_2$ the midpoints of $y_1$,$y'_1$ and $y_2$,$y'_2$, respectively. By the Busemann convexity it follows that $d(x_1,x_2) = \mbox{dist}(K_1,K_2)$. Since the metric in $X$ is convex, for every $z\in K_2$ we have
\begin{align*}
d(z,x_1) &\leq \frac{1}{2} d(z,y_1) + \frac{1}{2} d(z,y'_1) \leq \frac{1}{2} (\delta(y_1,K_2)+\delta(K_1,K_2))\\
& \leq \frac{(1+\lambda)}{2} \delta(K_1,K_2).
\end{align*}
Similarly, for every $z \in K_1$ we have $d(z,x_2) \leq \frac{(1+\lambda)}{2} \delta(K_1,K_2).$
Thus, there exists a pair of proximal points $(x_1,x_2)\in K_1 \times K_2$ and $\alpha \in (0,1)$ satisfying
$$\delta(x_1,K_2) \leq \alpha \delta(K_1,K_2) \text{ and } \delta(x_2,K_1)\leq \alpha \delta(K_1,K_2).
$$
Now consider the sets $L_1 \subseteq K_1$ and $L_2 \subseteq K_2$ defined as
\begin{align*}
 L_1 = \{ x\in K_1 : \ & \delta(x,K_2) \leq \alpha \delta(K_1,K_2) \text{ and for its proximal point }  y \in K_2,\\
& \delta(y,K_1)\leq \alpha \delta(K_1,K_2)\},
\end{align*}
\begin{align*}
 L_2 = \{ y\in K_2 : \ & \delta(y,K_1) \leq \alpha \delta(K_1,K_2) \text{ and for its proximal point }  x \in K_1,\\
& \delta(x,K_2)\leq \alpha \delta(K_1,K_2)\}.
\end{align*}
Since $x_1 \in L_1$ and $x_2 \in L_2$, $L_i\neq\emptyset$ for $i=1,2$. Next we show that $L_i$ is closed and convex for $i=1,2$. We just give the details for $L_1$ since for $L_2$ the proof follows similar patterns.
Let $\{v_n\} \subseteq L_1$ be a sequence that converges to a point $v\in K_1$. Since $d(v_n,z)\leq \alpha \delta(K_1,K_2)$ for every $n\in \mathbb{N}$ and $z\in K_2$, we get $\delta(v,K_2)\leq \alpha \delta(K_1,K_2)$. The fact that $v_n\in L_1$ implies
\begin{equation}\label{1}
\delta ( w_n,K_1)\leq \alpha \delta(K_1,K_2),
\end{equation}
where $w_n\in K_2$ is the proximal point of $v_n \in K_1$. Let $w \in K_2$ such that $d(v,w)=\text{dist}(K_1,K_2)$. By the Busemann convexity, we get
$$d\bigg{(}\frac{v_n+v}{2},\frac{w_n+w}{2}\bigg{)}=\text{dist}(K_1,K_2).$$
Now, by Proposition \ref{para}, we have $d(w_n,w)=d(v_n,v)$ for $n \in \mathbb{N}$, from where $w_n \to w$. Taking limit
in (\ref{1}), we may conclude $\delta(w,K_1)\leq \alpha \delta(K_1,K_2)$.
Consequently $v\in L_1$ and then $L_1$ is closed. Let $p_1,q_1\in L_1$. Next we see that $m_1=\frac{p_1+q_1}{2} \in L_1$. Let $p_2,q_2 \in K_2$ be the proximal points of $p_1$ and $q_1$, respectively. Consider $m_2=\frac{p_2+q_2}{2}$. Since $(K_1,K_2)$ is proximal and the space is Busemann convex, $d(m_1,m_2)=\text{dist}(K_1,K_2)$.
Let $z \in K_2$. The convexity of the metric implies
$$d(z,m_1)\leq \frac{1}{2}d(p_1,z)+\frac{1}{2}d(q_1,z)\leq \alpha \delta(K_1,K_2).$$
Thus, $\delta(m_1,K_2)\leq \alpha \delta(K_1,K_2)$. The fact that $\delta(m_2,K_1)\leq \alpha \delta(K_1,K_2)$ follows similarly since $\delta(p_2,K_1)$ and $\delta(q_2,K_1)$ are both $\leq \alpha \delta(K_1,K_2)$. Then $m_1 \in L_1$ and so $L_1$ is convex.

\noindent From $d(x_1,x_2)=\text{dist}(K_1,K_2)$ we get $\text{dist}(L_1,L_2)=\text{dist}(A_0,B_0)$. Moreover, by the definition of the sets $L_i$ with $i=1,2$, it is immediate that $(L_1,L_2)$ is a proximal pair.

In the sequel we see that $T(L_1)\subseteq L_2$ and $T(L_2) \subseteq L_1$. Let $x \in L_1$ and $y\in L_2$ such that $d(x,y)=\text{dist}(L_1,L_2)$. We prove that $Tx \in L_2$. Let $z \in K_2$. Since $d(Tx,Tz)\leq d(x,z)\leq \delta(x,K_2)\leq \alpha \delta(K_1,K_2)$, we get
$$T(K_2)\subseteq B(Tx,\alpha \delta(K_1,K_2))\cap K_1:=K^\prime_1.$$
Then $K^\prime_1$ is closed, convex and nonempty. Let $K^\prime_2\subseteq K_2$ be the set defined as
$$K^\prime_2=\{y^\prime \in K_2 : \text{ there exists } x^\prime \in K^\prime_1 \text{ with } d(x^\prime,y^\prime)=\text{dist}(K_1,K_2)\}.$$
Similarly as we proved before that $L_1$ is closed and convex, we get that $K^\prime_2$ is closed, convex and nonempty.

Now we see that $T(K^\prime_1)\subseteq K^\prime_2$ and $T(K^\prime_2)\subseteq K^\prime_1$. The fact that $T(K^\prime_2)\subseteq K^\prime_1$ is immediate. Let $p \in K^\prime_1$ and $q\in K^\prime_2$ such that $d(p,q)=\text{dist}(K^\prime_1,K^\prime_2)$. Then $d(Tp,Tq)=\text{dist}(K^\prime_1,K^\prime_2)$. Since $q\in K^\prime_2$, we have $Tq \in K^\prime_1$ and therefore $Tp \in K^\prime_2$. Consequently, $T(K^\prime_1)\subseteq K^\prime_2$. Notice that, by definition, the pair $(K^\prime_1,K^\prime_2)$ is also proximal and satisfies $\text{dist}(K^\prime_1,K^\prime_2)=\text{dist}(K_1,K_2)$. Thus,
$K^\prime_1 \cup K^\prime_2 \in \Gamma$ and by minimality of $K$ it follows that $K^\prime_1=K_1$ and $K^\prime_2=K_2$. Consequently,
$K_1 \subseteq B(Tx, \alpha \delta(K_1,K_2))$ and therefore $\delta(Tx,K_1) \leq \alpha \delta(K_1,K_2)$. To conclude that $Tx \in L_2$ it remains to see that the proximal point $z\in K_1$ of $Tx \in K_2$ satisfies $\delta(z,K_2) \leq \alpha \delta(K_1,K_2)$. Since $T$ is relatively nonexpansive, we have that $z=Ty$. Thus, we only need to show that $\delta(Ty,K_2) \leq \alpha \delta(K_1,K_2)$. However, notice that this inequality holds if we repeat the previous
construction of $K^\prime_1$ and $K^\prime_2$ starting from the point $y\in L_2$ and considering any point $z \in K_1$. Thus, we have
$Tx \in L_2$ and therefore $T(L_1) \subseteq L_2$. In a similar way, we may see that $T(L_2) \subseteq L_1$. As a consequence, $L_1 \cup L_2 \in \Gamma$. Since, $\delta(L_1,L_2)\leq \alpha \delta(K_1,K_2)$, we get a contradiction with the minimality of $K$.\end{proof}

\begin{theorem}\label{generaliza2} Let $X$ be a reflexive and Busemann convex metric
space and let $(A,B)$ be a nonempty closed convex pair of subsets of $X$ such that $A$ is bounded. Let
$T : A \cup B \rightarrow A \cup B$ be a noncyclic relatively nonexpansive mapping. Suppose $(A,B)$ has proximal normal structure. Then there exists a pair $(x,y) \in A \times B$ such that $x=Tx$, $y=Ty$ and $d(x,y)=\emph{dist}(A,B)$.
\end{theorem}

\begin{proof} Proceeding as in the previous theorem, we may see that $(A_0,B_0)$ is proximal, closed, convex and nonempty. Moreover, it is also immediate that $\text{dist}(A_0,B_0)=\text{dist}(A,B)$, $T(A_0) \subseteq A_0$ and $T(B_0) \subseteq B_0$. Now let $\Gamma$ be the collection of sets $F \subseteq A_0 \cup B_0$ such that $F \cap A_0$ and $F \cap B_0$ are closed, convex and nonempty and satisfy
$T(F \cap A_0) \subseteq F \cap A_0$ and $T(F \cap B_0) \subseteq F \cap B_0$, $\text{dist}(F \cap A_0,F \cap B_0)=\text{dist}(A_0,B_0)$ and the pair $(F \cap A_0, F \cap B_0)$ is proximal. Since $A_0 \cup B_0 \in \Gamma$, $\Gamma \neq \emptyset$.

Let $\{F_\alpha\}_{\alpha \in I}$ be a decreasing chain in $\Gamma$. Following similar patterns as in the previous proof, we get that $F_0= \cap_{\alpha \in I} F_\alpha \in \Gamma$. Then, applying Zorn's Lemma, we find a minimal element $K$ in $\Gamma$.

Let $K_1=K \cap A_0$ and $K_2=K \cap B_0$. If $\delta(K_1,K_2)=\text{dist}(K_1,K_2)$, then $d(p,q)=\text{dist}(K_1,K_2)$ for every pair $(p,q)\in K_1 \times K_2$. Let $(p,q)\in K_1 \times K_2$. Since $T$ is relatively nonexpansive, $d(Tp,Tq)=\text{dist}(K_1,K_2)=\text{dist}(A_0,B_0)$. Let $m\in K_2$ be the midpoint between $q$ and $Tq$. Then, since $d(m,Tp)=d(m,p)=\text{dist}(A_0,B_0)$ and $X$ is strictly convex, we get that $Tp=p$ and $Tq=q$ with $d(p,q)=\text{dist}(A_0,B_0)$, so that the result holds.
Suppose now $\delta(K_1,K_2)>\text{dist}(K_1,K_2)$. Repeating the reasoning of the previous theorem, we may find a pair of proximal points $(x_1,x_2) \in K_1 \times K_2$ such that $$\delta(x_1,K_2) \leq \alpha \delta(K_1,K_2) \text{ and } \delta(x_2,K_1)\leq \delta(K_1,K_2).
$$
We consider now the sets $L_1 \subseteq K_1$ and $L_2 \subseteq K_2$ defined as
\begin{align*}
 L_1 = \{ x\in K_1 : \ & \delta(x,K_2) \leq \alpha \delta(K_1,K_2) \text{ and for its proximal point }  y \in K_2,\\
& \delta(y,K_1)\leq \alpha \delta(K_1,K_2)\},
\end{align*}
\begin{align*}
 L_2 = \{ y\in K_2 : \ & \delta(y,K_1) \leq \alpha \delta(K_1,K_2) \text{ and for its proximal point }  x \in K_1,\\
& \delta(x,K_2)\leq \alpha \delta(K_1,K_2)\}.
\end{align*}
Since the definition of these sets is as in Theorem \ref{generaliza}, we have that $(L_1,L_2)$ is closed, convex, nonempty, proximal and satisfies  $\text{dist}(L_1,L_2)=\text{dist}(K_1,K_2)$. To see that $T(L_1)\subseteq L_1$ and $T(L_2) \subseteq L_2$ we may follow a similar reasoning to the one considered in Theorem \ref{generaliza} where the cyclic inclusion is proved. Although we omit some technical details, we include the proof for completeness. Let $x \in L_1$ and $y\in L_2$ such that $d(x,y)=\text{dist}(L_1,L_2)$. We prove that $Tx \in L_1$. Let $z \in K_2$. Since $d(Tx,Tz)\leq d(x,z)\leq \delta(x,K_2)\leq \alpha \delta(K_1,K_2)$, we get
$$T(K_2)\subseteq B(Tx,\alpha \delta(K_1,K_2))\cap K_2:=K^\prime_2.$$
Then $K^\prime_2$ is closed, convex and nonempty. Let $K^\prime_1\subseteq K_1$ be the set
$$K^\prime_1=\{x \in K_1 : \text{ there exists }  y \in K^\prime_2 \text{ with } d(x,y)=\text{dist}(K_1,K_2)\}.$$
Then $K^\prime_2$ is closed, convex and nonempty. Moreover, $(K^\prime_1,K^\prime_2)$ is proximal and satisfies $\text{dist}(K^\prime_1,K^\prime_2)$ $=$ $\text{dist}(K_1,K_2)$ and $T(K^\prime_1)\subseteq K^\prime_1$ and $T(K^\prime_2)\subseteq K^\prime_2$. Therefore,
$K^\prime_1 \cup K^\prime_2 \in \Gamma$ and by minimality of $K$ it follows that $K_2 \subseteq B(Tx, \alpha \delta(K_1,K_2))$ and therefore $\delta(Tx,K_2) \leq \alpha \delta(K_1,K_2)$. Proceeding similarly, we may see that $\delta(Ty,K_1) \leq \alpha \delta(K_1,K_2)$. Since $Ty \in K_2$ is the proximal point of $Tx \in K_1$, we conclude that
$Tx \in L_1$ and therefore $T(L_1) \subseteq L_1$. Similarly, $T(L_2) \subseteq L_2$. As a consequence, $L_1 \cup L_2 \in \Gamma$. Since, $\delta(L_1,L_2)\leq \alpha \delta(K_1,K_2)$, we get a contradiction with the minimality of $K$.\end{proof}

As a consequence of any of the two previous results, we get Kirk's fixed point theorem in the setting of reflexive and Busemann convex metric spaces when $\text{dist}(A,B)=0$. Notice that, in this particular case, the fact that $(A,B)$ has proximal normal structure implies that $A\cap B$ has normal structure in the sense of Brodski and Milman \cite{brmi}.

\begin{proposition} Every closed convex pair in a uniformly convex metric space $X$ has proximal normal structure.

\end{proposition}

\begin{proof} Let $(H_1,H_2)$ be a closed convex bounded proximal pair in $X$ with $\delta(H_1,H_2)$ $>$ $\text{dist}(H_1,H_2)$. Let $x,y\in H_1$ such that $d(x,y)>0$. Consider the points $x^\prime,y^\prime \in H_2$ such that $d(x,x^\prime)=d(y,y^\prime)=\text{dist}(H_1,H_2)$. Let $m=\frac{x+y}{2}\in H_1$ and $m^\prime=\frac{x^\prime+y^\prime}{2}\in H_2$. Since $X$ is strictly convex, $d(x^\prime,y^\prime)>0$. Let $\varepsilon = \min \{d(x,y), d(x^\prime,y^\prime)\}$ and $z \in H_2$. Denote by $\delta_X$ a modulus of uniform convexity of $X$. Then,
$$
d(z,m)\leq \bigg{(}1-\delta_X\bigg{(}\delta(H_1,H_2),\frac{\varepsilon}{\delta(H_1,H_2)}\bigg{)}\bigg{)} \delta(H_1,H_2).
$$
Similarly, if we take $z\in H_1$, we get
$$d(z,m^\prime)\leq \bigg{(}1-\delta_X\bigg{(}\delta(H_1,H_2),\frac{\varepsilon}{\delta(H_1,H_2)}\bigg{)}\bigg{)} \delta(H_1,H_2).
$$
Thus, $\delta(m,H_2) \leq \alpha \delta(H_1,H_2)$ and $\delta(m^\prime,H_1) \leq \alpha \delta(H_1,H_2)$,
for $\alpha=1-\delta_X\bigg{(}\delta(H_1,H_2),\frac{\varepsilon}{\delta(H_1,H_2)}\bigg{)}$.\end{proof}

\begin{corollary} Let $(A,B)$ be a closed convex pair in a complete Busemann convex metric space $X$. Suppose that $X$ is uniformly convex with a monotone or lower
semicontinuous from the right modulus of uniform convexity and $B$ is bounded. Let $T : A \cup B \rightarrow A \cup B$ be a cyclic relatively nonexpansive mapping. Then there exists a pair $(x,y) \in A \times B$ such that $d(x,Tx)=d(y,Ty)=\emph{dist}(A,B)$.

\end{corollary}

\begin{corollary}\label{coanterior} Let $(A,B)$ be a closed convex pair in a complete Busemann convex metric space $X$. Suppose that $X$ is uniformly convex with a monotone or lower
semicontinuous from the right modulus of uniform convexity and $B$ is bounded. Let $T : A \cup B \rightarrow A \cup B$ be a noncyclic relatively nonexpansive mapping. Then there exists a pair $(x,y) \in A \times B$ such that $x=Tx$, $y=Ty$ and $d(x,y)=\emph{dist}(A,B)$.
\end{corollary}

\begin{proposition} Let $(A,B)$ be a closed convex pair in a complete Busemann convex metric space $X$. Suppose that $X$ is uniformly convex with $\delta_X$ being a monotone or lower semicontinuous from the right modulus of uniform convexity and $B$ is bounded. Let $T : A \cup B \rightarrow A \cup B$ be a noncyclic relatively nonexpansive mapping. Let $x_0\in A_0$ and define $x_{n+1}=\frac{x_n+Tx_n}{2}$ for every $n \geq 1$. Then $\lim_n d(x_n,Tx_n)=0$.
Moreover, if $T(A)\subseteq C$, where $C$ is a compact set in $X$, then $\{x_n\}$ converges to a fixed point of $T$.
\end{proposition}

\begin{proof} By Corollary \ref{coanterior} we can find a point $y\in B_0$ such that $y=Ty$. Since the metric of the space is convex, we get that $\{d(y,x_n)\}$ is nonincreasing and so convergent to some $d \geq 0$.
Suppose first that $d=0$. In this case, the result is immediate since $\{Tx_n\}$ also converges to $y$.
Now we consider $d>0$. Suppose that there exists a subsequence $\{x_{n_k}\}$ of $\{x_n\}$ such that $d(x_{n_k},Tx_{n_k})\geq \varepsilon>0$ for every $k\geq 0$.

Suppose first $\delta_X$ is monotone. Let $0<\rho<\min\bigg{\{}\frac{d \delta_X(d+1,\frac{\varepsilon}{d+1})}{1-\delta_X(d+1,\frac{\varepsilon}{d+1})},1\bigg{\}}$. Then, from the uniform convexity of the space, there exists $k_0 \in \mathbb{N}$ such that
$$
d(y,x_{n_k+1})\leq \bigg{(}1-\delta_X\bigg{(}d+\rho,\frac{\varepsilon}{d+1}\bigg{)}\bigg{)} (d+\rho) \text{ for every } k \geq k_0.
$$
By the definition of $\rho$, we have $d(y,x_{n_k+1})<d$ for every $k \geq k_0$, which is a contradiction.

Suppose now $\delta_X$ is lower semicontinuous from the right. In this case, let $\varepsilon^*>0$ such that $0<\varepsilon^*<\delta_X(d,\frac{\varepsilon}{d+1})$. For such $\varepsilon^*>0$, consider $\mu(\varepsilon^*)>0$ such that $\delta_X(d,\frac{\varepsilon}{d+1})\leq \delta_X(r,\frac{\varepsilon}{d+1})+\varepsilon^*$ for every $r\in (d,d+\mu(\varepsilon^*))$. Let $0<\rho<\min\bigg{\{}\frac{d \delta_X(d,\frac{\varepsilon}{d+1})-d\varepsilon^*}{1-\delta_X(d,\frac{\varepsilon}{d+1})+\varepsilon^*},1,\mu(\varepsilon^*)\bigg{\}}$. By using the uniform convexity as before, there exists $k_0 \in \mathbb{N}$ such that
$$
d(y,x_{n_k+1})\leq \bigg{(}1-\delta_X\bigg{(}d+\rho,\frac{\varepsilon}{d+1}\bigg{)}\bigg{)} (d+\rho) \text{ for every } k \geq k_0.$$ Similarly
we get a contradiction. The rest of the proof follows similar patterns to those given in \cite[Proposition 2.3]{elvek}.\end{proof}

Next we provide a bound for the existence of approximate fixed points for the mapping $T$. Recall that having a metric space $(X,d)$, a mapping $T : X \to X$ and $\{x_n\} \subseteq X$, a mapping $\Phi : (0, \infty) \to \mathbb{N}$ is called an {\it approximate fixed point bound for $\{x_n\}$} (see also \cite{Koh05}) if
\[\forall \varepsilon > 0, \exists n \le \Phi(\varepsilon) \mbox{ such that } d(x_n, Tx_n) \le \varepsilon.\]
We don't include the proof of the result below since it can be obtained by following a similar reasoning as in the main result of \cite{leus}.

\begin{proposition}
Let $(A,B)$ be a closed convex pair in a complete Busemann convex metric space $X$. Suppose that $X$ is uniformly convex with $\delta_X$ being a monotone modulus of uniform convexity and $B$ is bounded. Let $T : A \cup B \to A \cup B$ be a noncyclic relatively nonexpansive mapping. Let $x_0 \in A_0$ and $b > 0$ such that there exists $y \in B_0$ with $y = Ty$ for which $d(x_0, y) \le b$. Define $x_{n+1} = \frac{x_n + Tx_n}{2}$ for every $n \ge 1$. Then $\Phi : (0, \infty) \to \mathbb{N}$,
\[\Phi(\varepsilon) = \left[\frac{2b}{\varepsilon \delta_X\left(b, \frac{\varepsilon}{b}\right)}\right]\]
is an approximate fixed point bound for $\{x_n\}$.
\end{proposition}

\begin{proposition} Let $(A,B)$ be a compact convex pair in a geodesic space with convex metric. Then $(A,B)$ has proximal normal structure.
\end{proposition}

\begin{proof} 
Let $(H_1,H_2) \subseteq (A,B)$ be a closed convex bounded and proximal pair in $X$ with $\delta(H_1,H_2)>\text{dist}(H_1,H_2)$. Suppose  $\delta(x,H_2)=\delta(H_1,H_2)$ for every $x\in H_1$. Let $x_0\in H_1$. Then there exists $y_0 \in H_2$ such that $\delta(x_0,H_2)=\delta(H_1,H_2)=d(x_0,y_0)$. Let $x_1\in H_1$ such that $d(x_1,y_0)=\text{dist}(H_1,H_2)$. Then $d(x_1,x_0)\geq d(x_0,y_0)-d(x_1,y_0)= \delta(H_1,H_2)-\text{dist}(H_1,H_2)$. Let $y_1\in H_2$ such that
$$d\bigg{(}\frac{x_1+x_0}{2},y_1\bigg{)}=\delta(H_1,H_2).$$
Since the metric is convex, the fact that
\[\delta(H_1,H_2)=d(y_1,\frac{x_1+x_0}{2})\leq \frac{1}{2} (d(y_1,x_0)+d(y_1,x_1))\]
implies $d(y_1,x_0)=d(y_1,x_1)=\delta(H_1,H_2)$. Let $m_{0,1}=\frac{x_1+x_0}{2}$. Take $x_2 \in H_1$ such that $d(y_1,x_2)=\text{dist}(H_1,H_2)$ and $y_2 \in H_2$ such that $d(y_2, \frac{1}{3}x_2+\frac{2}{3}m_{0,1})=\delta(H_1,H_2)$. Let $m_{1,2}=\frac{1}{3}x_2+\frac{2}{3}m_{0,1}$. By using again the convexity of the metric,
we obtain $\delta(H_1,H_2)=d(y_2,x_2)=d(y_2,m_{0,1})=d(y_1,m_{0,1})$. Moreover, this last equality also implies
$$d(y_2,x_0)=d(y_2,x_1)=\delta(H_1,H_2).$$
Suppose we have $\{x_1,\ldots,x_n\}$ in $H_1$, $\{m_{0,1}, m_{1,2},\ldots, m_{n-1,n}\}$ in $H_1$, where
$$m_{i-1,i} = \frac{1}{i+1} x_i + \frac{i}{i+1} m_{i-2,i-1}$$ for every $i\geq 2$, and $\{y_1,\ldots,y_{n-1}\}$ in $H_2$ such that
$$d(x_{i+1},y_i)=\text{dist}(H_1,H_2) \text{ for every } i=1\ldots n-1,
$$
$$
d(y_i,m_{i-1,i})=\delta(H_1,H_2) \text{ for every } i=1\ldots n-1  $$
and
$$d(y_i,x_j)=\delta(H_1,H_2) \text{ for every } i=0\ldots n-1 \text{ and } 0\leq j\leq i.
$$
Now consider the point $y_n\in H_2$ such that
$$d(y_n,m_{n-1,n})=\delta(H_1,H_2).
$$
Take $x_{n+1}\in H_1$ such that $$d(x_{n+1},y_n)=\text{dist}(H_1,H_2)
.$$
By using again the convexity of the metric, we may see that
$$d(y_n,x_i)=\delta(H_1,H_2) \text{ for every } i=0\ldots n. $$
As a consequence,
$$d(x_{n+1},x_i)\geq d(x_i,y_n)-d(x_{n+1},y_n)=\delta(H_1,H_2)-\text{dist}(H_1,H_2)$$
for every $n\in \mathbb{N}$ and for every $i=1\ldots n.$
Finally, by considering a convergent and, therefore, Cauchy subsequence of $\{x_n\}$ we get a contradiction. \end{proof}

\section{CAT($0$) spaces}

\begin{proposition} Let $(A,B)$ be a closed convex pair in a CAT(0) space. Consider the mapping $P : A \cup B \rightarrow A \cup B$ defined as
$$P(x) = \left\{
\begin{array}{c l}
 P_B(x)  & x\in A,\\
P_A(x) & x\in B.
\end{array}
\right.
$$
Then $P$ is a cyclic relatively nonexpansive mapping.
\end{proposition}

\begin{proof} The fact that $P$ is cyclic is immediate. Let $x\in A$ and $y\in B$. For simplicity, denote $a=d(x,P_Ay)$, $b=d(y,P_Bx)$, $c=d(y,P_Ay)$, $e=d(x,P_Bx)$, $h=d(x,y)$ and $r=d(P_Ay,P_Bx)$. Next we prove that $r\leq h$.

Let $\alpha = \angle_{P_Bx}(x,y)$ and $\beta=\angle_{P_Ay}(x,y)$. By Proposition \ref{projection}, $\cos \alpha, \cos \beta \leq 0$. By the Cosine Law in CAT($0$) spaces, we get
$$d(x,y)^2=h^2\geq e^2+b^2 \text{ and } d(x,y)^2=h^2\geq a^2+c^2.
$$
If we apply Theorem \ref{charact} to the four points $\{x,P_Bx, y, P_Ay\}$ and consider the two previous inequalities, we obtain
$$
h^2+r^2\leq a^2+b^2+c^2+e^2\leq 2 h^2,
$$ and the result follows. \end{proof}

As a consequence of the previous result, we may reason as in \cite{elvek} to conclude that in the setting of CAT($0$) spaces Theorem \ref{generaliza2} is a consequence of Theorem \ref{generaliza}.

Since every pair of closed and convex sets in a CAT($0$) space satisfies property $UC$ (see \cite{esfe2} for more details on this property), we may assert that every noncyclic or cyclic relatively nonexpansive mapping is also continuous if the pair $(A,B)$ is in addition proximal. Next we see that, as it happens in Hilbert spaces \cite[Proposition 3.2]{elvek}, a noncyclic relatively nonexpansive mapping is even nonexpansive if the pair $(A,B)$ is proximal.

\begin{proposition} Let $(A,B)$ be a closed convex bounded proximal pair of sets in a CAT(0) space. Let $T : A \cup B \rightarrow A \cup B$ be a noncyclic relatively nonexpansive mapping. Then $T$ is nonexpansive.

\end{proposition}

\begin{proof} Let $x_1,x_2 \in A$. Let $d_1=d(x_1,P_Bx_2)$, $d_2=d(x_2,P_Bx_1)$,
\[\alpha = \angle_{P_BTx_1}(Tx_1,TP_Bx_2) \quad \text{and} \quad \beta=\angle_{P_BTx_2}(Tx_2,TP_Bx_1).\]
Note that $d(x_1,x_2)=d(P_Bx_1,P_Bx_2)$, $d(Tx_1,Tx_2)=d(P_BTx_1,P_BTx_2)$, $d(x_1,P_Bx_1)=d(x_2,P_Bx_2)$ $=\text{dist}(A,B)$ and $\alpha, \beta \geq \pi/2$.

Let $\triangle_1=\triangle(TP_Bx_2,TP_Bx_1,Tx_1)$ and $\triangle_2=\triangle(TP_Bx_1,TP_Bx_2,Tx_2)$. Since $d(Tx_1, TP_Bx_1) \le d(x_1, P_Bx_1)$ it follows by Proposition \ref{projection}, ($1$) that $TP_Bx_1 = P_BTx_1$. Similarly, $TP_Bx_2 = P_BTx_2$. If we apply the Cosine Law in CAT($0$) spaces to
$\triangle_1$ and $\triangle_2$, we obtain
\begin{align*}
d(Tx_1,Tx_2)^2+\text{dist}(A,B)^2 & = d(P_BTx_1,P_BTx_2)^2+ \text{dist}(A,B)^2 \\
& \leq d(Tx_1,TP_B x_2)^2 \leq d_1^2
\end{align*}
and
\begin{align*}
d(Tx_1,Tx_2)^2+\text{dist}(A,B)^2 & = d(P_BTx_1,P_BTx_2)^2+ \text{dist}(A,B)^2 \\
& \leq d(Tx_2,TP_B x_1)^2\leq d_2^2.
\end{align*}
Thus,
\begin{equation}\label{2}
d(Tx_1,Tx_2)^2+\text{dist}(A,B)^2 \leq \min\{d_1^2, d_2^2\}.
\end{equation}

If we apply Theorem \ref{charact} to the four points $\{x_1,P_Bx_1, x_2, P_Bx_2\}$, we obtain

\begin{equation}\label{3}
2 \min\{d_1^2,d_2^2\} \leq d_1^2+d_2^2\leq 2(d(x_1,x_2)^2+\text{dist}(A,B)^2).
\end{equation}

By using (\ref{2}) and (\ref{3}), we get that $T$ is nonexpansive on $A$. In a similar way, we get that $T$ is nonexpansive on $B$ and then the result holds.\end{proof}

Next we see that a similar result also holds for cyclic relatively nonexpansive mappings.

\begin{proposition} Let $(A,B)$ be a closed convex bounded proximal pair of sets in a CAT(0) space. Let $T : A \cup B \rightarrow A \cup B$ be a cyclic relatively nonexpansive mapping. Then $T$ is nonexpansive.

\end{proposition}

\begin{proof}  Let $x_1,x_2 \in A$. Let $d_1=d(x_1,P_Bx_2)$, $d_2=d(x_2,P_Bx_1)$
\[\alpha = \angle_{P_ATx_1}(Tx_1,TP_Bx_2) \quad \text{and} \quad \beta=\angle_{P_ATx_2}(Tx_2,TP_Bx_1).\]
Note that $d(x_1,x_2)=d(P_Bx_1,P_Bx_2)$, $d(Tx_1,Tx_2)=d(P_BTx_1,P_BTx_2)$, $d(x_1,P_Bx_1)=d(x_2,P_Bx_2)$ $=\text{dist}(A,B)$ and $\alpha, \beta \geq \pi/2$.

Let $\triangle_1=\triangle(TP_Bx_2,TP_Bx_1,Tx_1)$ and $\triangle_2=\triangle(TP_Bx_1,TP_Bx_2,Tx_2)$. By Proposition \ref{projection}, ($1$), we get that $TP_Bx_1=P_ATx_1$ and $TP_Bx_2=P_ATx_2$. By applying now the Cosine Law to
$\triangle_1$ and $\triangle_2$ and Theorem \ref{charact} to $\{x_1,P_Bx_1, x_2, P_Bx_2\}$ as in the previous theorem we get the result. \end{proof}

\section{Proximal normal structure: a sufficient but not necessary condition}

In 1979, Karlovitz \cite{kar} proved that the normal structure of the domain of a nonexpansive self-mapping $T$ is a sufficient but not necessary condition to guarantee existence of fixed points of such a mapping in the context of reflexive Banach spaces. For this aim, a very specific family of reflexive spaces which originated with R. C. James was considered. In the same setting we see now that the proximal normal structure behaves similarly with respect to relatively nonexpansive mappings. First we give an example of a closed convex bounded pair of sets in a reflexive Banach space that does not have proximal normal structure.

\begin{example}

Let $X$ denote the Banach space given by the set $\ell_2$ endowed with the norm
$$\|x\|=\max\{\|x\|_\infty, \|x\|_2/ \sqrt{2}\}.$$
Let
$$
A=B(\theta,1) \cap \{x = \{x_n\}_{n\geq1} \in X : x_1 = 1, x_i\geq 0  \text{ for every } i\geq1\}$$ and
$$B=B(\theta,2) \cap \{x = \{x_n\}_{n\geq1}  \in X : x_1 = 2, x_i\geq 0 \text{ for every } i\geq1\},
$$where $\theta$ denotes the origin of the space $\ell_2$.
It is easy to see that $(A,B)$ is closed, convex, bounded and proximal. From the definition we have $\|x-y\|\geq 1$ for every $x\in A$ and $y\in B$. Notice that $e_1\in A$, $2e_1\in B$ and $\|e_1-2e_1\|=1$. Then $\mbox{dist}(A,B) = 1$. We claim $\delta(A,B)=2$. Let $x=\{x_i\} \in A$ and $y=\{y_i\} \in B$. Since $x\in A$, $0\leq x_i \leq 1$ for every $i\geq 1$. Equally, $y\in B$ implies $0\leq y_i \leq 2$. Thus, $\|x-y\|_\infty \leq 2$. Note that $\|x-y\|_2^2=\sum_{i=1}^\infty x_i^2 + \sum_{i=1}^\infty y_i^2-2(x_1 y_1)-2 \sum_{i=2}^\infty x_i y_i \leq 6$. Then $\|x-y\|_2/ \sqrt{2}\leq \sqrt{3} < 2$ and therefore $\|x-y\| \leq 2$. It is easy to see that $\{e_1+e_n\}_{n\geq 2} \subseteq A$ and $\{2e_1+2e_n\}_{n \geq 2} \subseteq B$. Since $d(e_1+e_n, 2e_1+2e_m)=2$ for $n\neq m$, we have $\delta(A,B)=2$. On the other hand, we have $\delta(x,B)=2$ for every $x=\{x_i\}_{i\geq1} \in A$ . This is a consequence of
$$\lim_{n\rightarrow \infty} d(x,2e_1+2e_n) \geq \lim_{n \rightarrow \infty}|x_n-2|=2.$$
Then $(A,B)$ does not have proximal normal structure.

\end{example}

\begin{remark} Let $A^*= \overline{\mbox{conv}}(\{e_1+e_n : n\geq 2\})$ and $B^*= \overline{\mbox{conv}}(\{2e_1+2e_n : n\geq 2\})$. Note that the pair $(A^*,B^*) \subseteq (A,B)$ does not have proximal structure either.
\end{remark}

Next we see that proximal normal structure is a sufficient but not necessary condition to obtain the existence of best proximity points in Theorem 2.1 in \cite{elvek} and therefore also in Theorem \ref{generaliza}.

\begin{proposition} Let $(A,B)$ be the pair of sets considered in the previous example and suppose $T : A \cup B \rightarrow A \cup B$ is a cyclic relatively nonexpansive mapping. Then there exists a pair $(x,y) \in A \times B$ such that $d(x,Tx)=d(y,Ty)=\emph{dist}(A,B)$.

\end{proposition}

\begin{proof} Consider the point $2e_1 \in B$. Notice that $d(2e_1,A)=\delta(2e_1,A)=1$. Thus,
\[1\leq d(T(2e_1),T^2(2e_1))\leq d(2e_1,T(2e_1))=1\]
and the result follows.
\end{proof}

\begin{remark} Notice that the previous result also holds whenever we have a cyclic relatively nonexpansive mapping defined on a pair $(A,K)\subseteq (A,B)$ with $2e_1 \in K$.
\end{remark}

\subsection*{Acknowledgements}

Aurora Fern\' andez-Le\' on was partially supported by the Plan Andaluz de Investigaci\' on de la Junta de Andalucía FQM-127 and Grant P08-FQM-03543, and by MEC Grant MTM2009-10696-C02-01. Part of this work was carried out while she was visiting the Babe\c s-Bolyai University in Cluj-Napoca. She acknowledges the kind hospitality of the Department of Mathematics.

Adriana Nicolae was supported by a grant of the Romanian National Authority for Scientific Research, CNCS-UEFISCDI, project number PN-II-ID-PCE-2011-3-0383.

\end{document}